\documentclass[11pt,a4paper]{amsart}
\usepackage[foot]{amsaddr}
\usepackage{texmac}
\usepackage{enumitem}
\usepackage{mathtools}
\usepackage{hyperref}
\usepackage[all,cmtip]{xy}

\operators{Stab,Prim,Inn,Out,ss,B,G,R,K,Inf,Def,vol,Ext,St,ab,ind,res,GL,nrd,PSL,PGL}

\bbsymbols{b}{F,Z,Q,P,R,C,H}
\calsymbols{c}{H,U,X,Y,P,M,D,C,F,O,S}
\fraksymbols{f}{H}

\let\le\leqslant
\let\ge\geqslant

\def\dsum_#1_#2{\sum_{{#1}\atop {#2}}}
\def\rar{\rightarrow}
\def\maxprime{71}

\newcommand{\modtors}[1]{{#1}^{\mathrm{fr}}}
\newcommand{\modStors}[1]{{#1}^{S\text{-}\mathrm{fr}}}
\newcommand{\Stors}{{S\text{-}\mathrm{tors}}}
\newcommand{\lquo}{\backslash}

\newcommand{\order}{\cO}
\newcommand{\mat}{\cM}
\newcommand{\prm}{\mathfrak{p}}
\newcommand{\hyp}{\fH}
\newcommand{\ideal}{\mathfrak{N}}

\begin{document}

\title{Torsion homology and regulators of isospectral manifolds}
\author{Alex Bartel}
\address{Mathematics Institute, Zeeman Building, University of Warwick,
Coventry CV4 7AL, UK}
\author{Aurel Page}
\email{a.bartel@warwick.ac.uk, a.r.page@warwick.ac.uk}
\thanks{The authors thank the EPSRC for financial support via a First Grant,
respectively the EPSRC Programme Grant EP/K034383/1 LMF: L-Functions and Modular
Forms}
\llap{.\hskip 10cm} \vskip -0.8cm

\maketitle

\begin{abstract}
Given a finite group $G$, a $G$-covering of closed Riemannian manifolds, and
a so-called $G$-relation, a construction of Sunada produces
a pair of manifolds $M_1$ and $M_2$ that are strongly isospectral.
Such manifolds have the same
dimension and the same volume, and their rational homology groups are isomorphic.
Here, we investigate the relationship between their integral homology.
The Cheeger--M\"uller Theorem implies that a certain product of orders of torsion
homology and of regulators for $M_1$ agrees with that for $M_2$.
We exhibit a connection between the torsion in the integral homology of $M_1$ and
$M_2$ on the one hand, and the $G$-module structure of integral homology of the
covering manifold on the other,
by interpreting the quotients $\Reg_i(M_1)/\Reg_i(M_2)$ representation
theoretically. Further, we prove that the
$p^\infty$-torsion in the homology of $M_1$ is isomorphic to that of $M_2$ for
all primes $p\nmid \#G$. For $p\leq \maxprime$, we give examples of pairs of
strongly isospectral hyperbolic $3$-manifolds for which the $p$-torsion homology
differs, and we conjecture such examples to exist for all primes~$p$.
\end{abstract}

\section{Introduction}\noindent
Two closed Riemannian manifolds $M_1$ and $M_2$ are said to be \emph{isospectral}
if the Laplace--Beltrami operators acting on functions on $M_1$ and on $M_2$
have the same spectrum,
equivalently if the spectral zeta functions $\zeta(M_1,s)$ and $\zeta(M_2,s)$
are equal. The manifolds are said to be \emph{strongly isospectral} if
the spectra of every natural self-adjoint elliptic differential operator
on~$M_1$ and~$M_2$ agree.
In particular, if $M_1$ and~$M_2$ are strongly isospectral, then
the Laplace--de Rham operators acting on differential $i$-forms on the respective
manifold have the same spectrum, equivalently the
zeta functions $\zeta_i(M_1,s)$ and $\zeta_i(M_2,s)$ encoding those spectra are
equal. Following Kac's famous question ``Can one hear the shape of a drum?''
\cite{Kac}, the following broad questions have received a lot of
attention~\cite{gordon, GordonSchueth}.
\begin{question}\label{q:1}
Which isometry invariants of closed Riemannian manifolds are
isospectral invariants? Which ones are strongly isospectral invariants?
\end{question}

For example, two Riemannian manifolds that are strongly isospectral
have the same dimension, volume and Betti numbers.

By analogy with the number theoretic notion of so-called arithmetically
equivalent number fields, Sunada \cite{Sunada} proposed a general
construction showing that there exist Riemannian manifolds
that are strongly isospectral, but are not even homeomorphic.
Let us briefly recall Sunada's
construction. Let $X\rar Y$ be a $G$-covering of closed 
manifolds, where
$G$ is a finite group.
Let~$U_1$, $U_2$ be two subgroups of $G$ with the property that there is an
isomorphism $\C[G/U_1]\cong \C[G/U_2]$ of linear permutation representations of $G$,
or equivalently that for every conjugacy class~$c$ of~$G$ we have~$\#(c\cap U_1)
= \#(c\cap U_2)$.
In this case we say that the formal linear combination $U_1-U_2$ is a
$G$-relation. Sunada proves that then the intermediate coverings $X/U_1$ and
$X/U_2$ are strongly isospectral.

In this paper, we want to take the analogy with number fields further. The
Cheeger-M\"uller Theorem \cite{Cheeger,Mueller1,Mueller2}, proving a conjecture of
Ray and Singer \cite{RaySinger}, gives a special value formula for the spectral
zeta functions of a Riemannian manifold, somewhat analogous to the analytic
class number formula for the Dedekind zeta function of a number field \cite{ClNoFormula}.
It implies that if $M_1$ and $M_2$ are strongly isospectral Riemannian manifolds, then
\begin{align}\label{eq:CheegerMuller}
\prod_{i=0}^d\left(\frac{\Reg_i(M_1)}{\#H_i(M_1,\Z)_{\tors}}\right)^{(-1)^i}=\;\;\;
\prod_{i=0}^d\left(\frac{\Reg_i(M_2)}{\#H_i(M_2,\Z)_{\tors}}\right)^{(-1)^i},
\end{align}
where, for a Riemannian manifold $M$, $\Reg_i(M)$ is the covolume of the
lattice~$H_i(M,\Z)/H_i(M,\Z)_{\tors}$ in the vector space $H_i(M,\bR)$ with respect to
a certain canonical inner product (see Notation \ref{not:pairings} \ref{not:reg}).
The study of torsion homology and regulators of manifolds has attracted a lot of
attention recently, see for instance~\cite{BVtorsion, BSVtorsion, brockdunfield, torsionJL}.

In light of equation~(\ref{eq:CheegerMuller}), a natural instance of
Question~\ref{q:1} is whether, for $M_1$ and $M_2$ as above, an  integer~$i\ge
0$, and a prime number $p$, we must necessarily have
\begin{align}\label{eq:ChMtors}
\#H_i(M_1,\Z)[p^\infty]=\#H_i(M_2,\Z)[p^\infty],
\end{align}
where $[p^\infty]$ denotes the $p$-primary torsion subgroup.

Here, we answer this and similar questions in the context of Sunada's
construction. More broadly, we study homological properties
of covering manifolds in $G$-relations, analogous to an old and fruitful
line of research in number theory.
For example, the following basic result is analogous to 
results of \cite{Boltje} on class groups of number fields.

%

\begin{prop}\label{prop:intro1}
Let $X\rar Y$ be a $G$-covering of closed Riemannian manifolds,
let $U_1 - U_2$ be a $G$-relation, and let $i\ge 0$ be an
integer. Then for all primes $p$ that do not divide
$\#G$, we have
$$
H_i(X/U_1,\Z)[p^\infty] \cong H_i(X/U_2,\Z)[p^\infty].
$$
\end{prop}
We actually prove a more general result, see Theorem \ref{thm:noMackeyFunctor}.

In providing counter examples, we have concentrated on closed
hyperbolic $3$-manifolds, a class of manifolds that plays an important r\^ole
in other branches of geometry, as well as in number theory. The following is
a result in the opposite direction to that of Proposition~\ref{prop:intro1}.

\begin{prop}\label{prop:intro2}
Let $p\leq \maxprime$ be a prime number. Then there exist strongly isospectral
closed hyperbolic $3$-manifolds $M_1$ and $M_2$ such that
\[
  \#H_1(M_1,\Z)[p^\infty]\neq \#H_1(M_2,\Z)[p^\infty].
\]
Moreover,
if~$2<p\leq \maxprime$, then there exist $3$-manifolds $M_1$
and $M_2$ as above with
$$
\#H_1(M_1,\Z)[p^\infty]=1,\;\;\text{ and }\;\;\#H_1(M_2,\Z)[p^\infty]=p.
$$
\end{prop}

Note that hyperbolic $3$-manifolds have torsion-free homology in degrees~$0$, $2$,
and $3$, so in the situation of Proposition \ref{prop:intro2}, equation
(\ref{eq:ChMtors}) is trivially satisfied in those degrees.

It is already known that for every prime $p$, there exist strongly isospectral closed Riemannian
$4$-manifolds $M_1$ and $M_2$ 
such that $\#H_1(M_1,\Z)[p^\infty]\neq
\#H_1(M_2,\Z)[p^\infty]$. This is shown in \cite{Sunada}, using the
following two results. Firstly, every finite group can be realised as the
fundamental group of a closed smooth 4-manifold \cite[p. 402]{Shafarevich}.
Secondly, for every prime $p$, there exists a finite $p$-group $G$ with a
$G$-relation $U_1 - U_2$, where the abelianisations of $U_1$ and $U_2$ have different
orders \cite[\S 1, Example 3]{Sunada}.

At the other extreme, if $M$ is a closed oriented 2-manifold, then
$H_1(M,\Z)$ is torsion-free. Riemannian $3$-manifolds present an in-between case.
It follows from the Elliptisation Theorem, as proven by Perelman
\cite{Per1,Per2,Per3}, and from a theorem of Ikeda \cite{Ikeda}, that any two
isospectral $3$-manifolds with finite fundamental groups are homeomorphic.
In particular, Sunada's
construction for 4-manifolds cannot possibly generalise to $3$-manifolds.
Instead, we prove Proposition \ref{prop:intro2} by a judicious choice of $G$-relation
$U_1-U_2$, a different one for each prime $p$, and
then by performing a computer search among $G$-coverings of hyperbolic
$3$-manifolds and computing the resulting torsion homology of the intermediate
coverings corresponding to $U_1$ and $U_2$. Nevertheless, we believe that the
evidence of Proposition \ref{prop:intro2} justifies the following
conjecture\footnote{A few months after submitting this paper, we 
established Conjecture~\ref{conj:intro} in~\cite{equivsurg} using techniques
introduced in this paper, see in particular Proposition~\ref{prop:impliesconj}}.

\begin{conjecture}\label{conj:intro}
For every prime number $p$, there exist strongly isospectral closed hyperbolic $3$-manifolds
$M_1$ and $M_2$ such that
$\#H_1(M_1,\Z)[p^\infty]\neq \#H_1(M_2,\Z)[p^\infty]$.
\end{conjecture}

It follows from equation (\ref{eq:CheegerMuller}) that if $M_1$ and $M_2$
are strongly isospectral, then $\prod_i \left(\Reg_i(M_1)/\Reg_i(M_2)\right)^{(-1)^i}$
can be expressed in terms of torsion homology of the two manifolds, and in
particular is a rational number. Moreover, if $M_1$ and $M_2$ arise from Sunada's
construction, then it follows from Proposition \ref{prop:intro1}
that the numerator and denominator of this rational number are only divisible
by prime divisors of~$\#G$. Like in the case of torsion, it would be
interesting to understand the quotient of regulators in each
degree separately. In this direction we have the following result, which we
now state for arbitrary $G$-relations (see Definition \ref{defn:relation}).
Below, for a group~$U$ we write $U^{\ab}= U/[U,U]$ where~$[U,U]$ denotes the derived subgroup of $U$.
\begin{theorem}\label{thm:intro3}
Let $X\rar Y$ be a $G$-covering of closed oriented Riemannian $d$-manifolds,
and let $\sum_j n_jU_j$, $n_j\in \Z$ be a $G$-relation. Then:
\begin{enumerate}
\item for every integer $i\ge 0$, the product
$\prod_j \Reg_i(X/U_j)^{2n_j}$ is a rational
number that is a product of powers of prime divisors of $\#G$;
\item\label{item:regcompQ} for every $i\geq 0$, the class of
$\prod_j \Reg_{i}(X/U_j)^{2n_j}$ in $\Q^\times/(\Q^\times)^2$ only depends
on the isomorphism class of the $G$-module $H_i(X,\Q)$;
\item\label{item:regcompZ} if $p$ is a prime number that does not divide
  $\# U_j^{\ab}$ for any~$j$, then the $p$-part of
$\prod_j \Reg_{d-1}(X/U_j)^{2n_j} = \prod_j \Reg_{1}(X/U_j)^{-2n_j}$ only depends on
the isomorphism class of the $G$-module $H_{d-1}(X,\Z_{(p)})$.
\end{enumerate}
\end{theorem}
The representation theoretic invariant that appears in parts
(\ref{item:regcompQ}) and (\ref{item:regcompZ}) of the theorem
was first introduced by Dokchitser--Dokchitser in the context of regulators
of elliptic curves (see e.g. \cite{tamroot}), and is called the
regulator constant of the representation with respect to the given $G$-relation.

We actually relate the quotients of regulators to regulator constants in greater
generality, namely without the hypothesis on $p$ in part (\ref{item:regcompZ}) of Theorem
\ref{thm:intro3}. But in general, there is an ``error term'', which measures
the failure of $G$-descent for homology in the covering --
see Theorem \ref{thm:regcomparison} for the precise statement, but see also
Remark \ref{rmrk:noclue}.

Vign\'eras~\cite{vigneras} introduced another well-known construction of
isospectral manifolds, of a more arithmetic flavour.
In~\cite{torsionJL}, Calegari and Venkatesh study
torsion homology in pairs of manifolds that are of a similar nature to the
examples of Vign\'eras,
called ``Jacquet--Langlands pairs'', that are not isospectral but whose spectra
are closely related. They raise, in~\cite[Section~7.10]{torsionJL}, the question
of giving algebraic interpretations to the quotients of torsion homology and
regulators separately, analogously to Theorem~\ref{thm:intro3}. It
would also be interesting to investigate the same type of questions in the
context of the Vign\'eras examples.

In Section \ref{sec:examples}, we give examples of how Theorem \ref{thm:intro3}
allows one to deduce concrete information about the $\Q[G]$-module
structure of $H_1(X,\Q)$ from the torsion homology of quotients of $X$. It also
allows one to formulate a possible representation theoretic
line of attack on Conjecture \ref{conj:intro}, such as the following result,
which will be proved as Corollary \ref{cor:impliesconj}.
\begin{proposition}\label{prop:impliesconj}
Let $p$ be an odd prime, let $G=\GL_2(\F_p)$, and define the subgroups
$$
B=\begin{pmatrix}\F_p^\times&\F_p\\
0&\F_p^\times\end{pmatrix}
,\;\;\;\;
U=\begin{pmatrix}(\F_p^\times)^2&\F_p\\
0&\F_p^\times\end{pmatrix}
$$
of $G$.
Suppose that there exists a $G$-covering $X\rar Y$ of closed hyperbolic
$3$-manifolds such that the difference of the Betti numbers $b_1(X/U) - b_1(X/B)$
is odd. Then Conjecture \ref{conj:intro} holds for~$p$.
\end{proposition}

To fully appreciate the representation theoretic nature of Proposition
\ref{prop:impliesconj}, see the full statement of Corollary \ref{cor:impliesconj}.

The structure of the paper is as follows. In Section \ref{sec:Brauer} we
recall the formalism of Burnside groups, representation groups, and $G$-relations,
and the definition of regulator constants. In Section
\ref{sec:manifolds} we investigate the behaviour of homological invariants
of Riemannian manifolds in $G$-relations, and prove Proposition~\ref{prop:intro1}
and Theorem~\ref{thm:intro3}. In Section \ref{sec:regconst} we compute the regulator
constants of the rational representations of $\GL_2(\F_p)$ for odd primes $p$,
and of $(\Z/8\Z)\rtimes (\Z/8\Z)^\times$, with respect to certain $G$-relations.
These computations will allow us to deduce concrete information
about the $\Q[G]$-module structure of the rational homology of $G$-coverings
from the torsion homology of intermediate coverings.
Together with Theorem \ref{thm:regcomparison},
these calculations will also imply Proposition \ref{prop:impliesconj}.
We prove Proposition~\ref{prop:intro2} by a direct computer search, and
Section~\ref{sec:calc} is devoted to the methods and algorithms used to that
end. In Section~\ref{sec:examples} we have collected some interesting examples,
illustrating the various phenomena that we investigate here.

All our Riemannian manifolds will be assumed to be finite-dimensional,
oriented, and closed. By an automorphism of a Riemannian manifold
we mean an orientation preserving diffeomorphism from the manifold to
itself that is a local isometry. By a hyperbolic $3$-manifold, we mean a
quotient of the hyperbolic $3$-space by a discrete subgroup of
orientation preserving isometries.
If $p$ is a prime number, we will write $\Z_{(p)}$ for the localisation
of~$\Z$ at $p$, i.e. the subring $\{\frac ab\colon p\nmid b\}$ of $\Q$.
For a rational number $x=p^n\frac ab$, with $n\in \Z$ and $p\nmid ab$,
the $p$-adic valuation of $x$ is defined to be $\ord_p(x) = n$.
When~$M$ is a~$\Z$-module, we write~$M_{\tors}$ for the torsion submodule of~$M$
and~$\modtors{M} = M/M_{\tors}$ for the torsion-free quotient of~$M$.

\begin{acknowledgements}
We would like to thank Nicolas Bergeron, Nathan Dunfield, Derek Holt, Emilio Lauret,
and Karen Vogtmann for helpful discussions, and the anonymous referee for
suggestions that improved the exposition.
\end{acknowledgements}

\section{$G$-relations and regulator constants}\label{sec:Brauer}
We recall some standard definitions, for which we refer to \cite[Ch. 11]{CR}. By
``module'' we will always mean a finitely generated left module.
\begin{definition}
Let $G$ be a finite group. The \emph{Burnside group} of $G$ is the free
abelian group on isomorphism classes of transitive
$G$-sets.
\end{definition}

The set of transitive $G$-sets is in bijection with the set
of conjugacy classes of subgroups of $G$ via the map that assigns to the
subgroup $U$ the set of cosets~$G/U$. Using this identification, we will
represent elements of the Burnside group as formal sums $\sum_j n_j U_j$, where
$n_j\in \bZ$ and $U_j\leq G$.

\begin{definition}
Let $R$ be a domain. The \emph{representation group} of~$G$ over~$R$
is the free abelian group on isomorphism classes of $R$-free
indecomposable~$R[G]$-modules, where $R[G]$ is the group ring of $G$ over $R$.
\end{definition}

Let $G$ be a finite group and $R$ be a domain. 
We write $\triv$ for the free $R$-module of rank~$1$ with trivial $G$-action.
We have a natural group homomorphism~$\Psi_{R[G]}$ from the Burnside group of
$G$ to the representation group of~$G$ over~$R$,
which sends a $G$-set $X$ to the permutation module $R[X]$ with
an~$R$-basis indexed by the elements of $X$, and with the $R$-linear $G$-action
given by permutations of the basis.

\begin{definition}\label{defn:relation}
Let $R$ be a domain.
An \emph{$R[G]$-relation} is an element of the kernel of $\Psi_{R[G]}$.
A $\Q[G]$-relation will be referred to simply as a $G$-relation.
\end{definition}

\begin{example}\label{ex:dihedral}
Let $p$ be a prime, and let
$$
G=\langle \sigma,\tau\colon\sigma^p=\tau^2=\id,\tau\sigma\tau=\sigma^{-1}\rangle
$$
be a dihedral group of order $2p$. Then $\Theta=1-2C_2-C_p+2G$ is a $G$-relation.
Concretely, this means that there is an isomorphism of linear permutation representations
of $G$ over $\Q$,
$$
\Q[G/1]\oplus \Q[G/G]^{\oplus 2}\cong \Q[G/C_2]^{\oplus 2} \oplus \Q[G/C_p].
$$
Moreover, for every prime $q\neq p$, the relation $\Theta$ is in fact a
$\Z_{(q)}[G]$-relation, but it is not a $\Z_{(p)}[G]$-relation -- see
\cite[Proof of Proposition 3.9]{Bartel}.
\end{example}

\begin{example}\label{ex:2-Gassman}
Let $G$ be the affine linear group over $\Z/8\Z$, i.e. the group of linear
transformations $T_{a,b}\colon x\mapsto ax+b$ of $\Z/8\Z$, where $a\in (\Z/8\Z)^{\times}$
and $b\in \Z/8\Z$. Consider the subgroups $U_1=\langle T_{a,0}\colon a \in (\Z/8\Z)^\times\rangle$
and $U_2= \langle T_{-1,0}, T_{3,4}\rangle$.
The group $G$ is isomorphic to the semi-direct product $\Z/8\Z\rtimes (\Z/8\Z)^\times$,
and the subgroups $U_1$ and $U_2$ are both isomorphic to $C_2\times C_2$.
Then $U_1-U_2$ is a $\Z_{(p)}[G]$-relation for every odd prime $p$, but is not
a $\Z_{(2)}[G]$-relation, as can be deduced from a direct character computation and
Lemma \ref{lem:Grelations} below.
\end{example}

\begin{example}\label{ex:p-Gassman}
Let $p$ be an odd prime, and
let $G=\GL_2(\bF_p)$. Consider the two subgroups
$$
U_1=\begin{pmatrix}\bF_p^{\times} & \bF_p \\ 0 & (\bF_p^{\times})^2\end{pmatrix},\;\;\;\;
U_2=\begin{pmatrix}(\bF_p^{\times})^2 & \bF_p \\ 0 & \bF_p^{\times}\end{pmatrix}.
$$
Then for every prime~$q$, $U_1-U_2$ is a $\Z_{(q)}[G]$-relation if and only
if~$q\neq p$.

Observe that $U_1\cap U_2$ contains the subgroup
$N=\{\left(\begin{smallmatrix}a&0\\ 0&a\end{smallmatrix}\right)\colon a\in (\F_p^\times)^2\}$,
which is central in $G$. Set $\bar{G}=G/N$, and $\bar{U}_j=U_j/N$ for $j=1$ and
$2$. We then get the $\bar{G}$-relation $\bar{U}_1-\bar{U}_2$.
Moreover, the triple $(\bar{G},\bar{U}_1,\bar{U}_1)$ minimises the index
$(\Gamma\colon S_1)$ among all triples $(\Gamma,S_1,S_2)$, where $\Gamma$
is a finite group, and $S_1-S_2$ is a $\Gamma$-relation that is not a
$\Z_{(p)}[\Gamma]$-relation (see \cite{deSmit1}).
\end{example}

\begin{lemma}\label{lem:Grelations}
Let $G$ be a finite group, let $p$ be a prime number not dividing~$\#G$, and
let $\Theta$ be an element of the Burnside group of $G$. Then the following are
equivalent:
\begin{enumerate}
\item $\Theta$ is a $\C[G]$-relation;
\item $\Theta$ is a $\Q[G]$-relation;
\item $\Theta$ is a $\Z_{(p)}[G]$-relation;
\item $\Theta$ is an $\F_p[G]$-relation.
\end{enumerate}
\end{lemma}
\begin{proof}
The equivalence of (1) and (2) follows from the fact that if~$\rho_1$
and $\rho_2$ are $\Q[G]$-modules, then
$\dim_{\Q}\Hom(\rho_1,\rho_2)=\dim_{\C}\Hom(\C\otimes_{\Q} \rho_1,\C\otimes_{\Q}\rho_2)$
(see \cite[\S 2e]{maxorders}).
The remaining equivalences are proven in \cite[Chapter 5]{Benson}.
\end{proof}

\begin{notation}\label{not:BurnsideFunctions}
Let $f$ be any function on the set of conjugacy classes of subgroups of a
finite group $G$ with values in an abelian group~$A$.
Then $f$ extends to a unique group homomorphism from the Burnside group
of $G$ to $A$, defined by $f(\sum_j n_jU_j)=\prod_j f(U_j)^{n_j}$, where
$U_j$ are subgroups of~$G$ and~$n_j\in \Z$.
\end{notation}

Let $S$ be a subring of $\bR$ that is a PID. The main examples we are thinking of
are $\Z$, the localisation $\Z_{(p)}$ where $p$ is a prime, and $\Q$.
Given an~$S[G]$-module $A$, let $A_{\Stors}$ denote the $S[G]$-submodule
consisting of $S$-torsion elements of $A$, and let $\modStors{A}=A/A_{\Stors}$.
Since $S$ is a PID, for every~$S[G]$-module $A$, the quotient~$\modStors{A}$
is an $S$-free $S[G]$-module.
Moreover, $\bR\otimes_S A$ is \emph{self-dual},
i.e. it is isomorphic to the $\bR[G]$-module $\Hom_{S}(A,\bR)$. Thus
there exists an $\bR$-valued $S$-bilinear $G$-invariant
non-degenerate pairing on $\modStors{A}$. From now on, we will just say \emph{pairing
on $A$} when we mean an $\bR$-valued $S$-bilinear $G$-invariant pairing on $A$ that
is non-degenerate on~$\modStors{A}$.

\begin{definition}
Let $G$ be a finite group, let $S$ a subring of $\bR$ that is a PID, let $R$
be a subring of $\Q$, and let $A$ be an $S[G]$-module.
Let $\langle\cdot,\cdot\rangle$ be a pairing on $A$, and
let $\Theta=\sum_j n_jU_j$ be an $R[G]$-relation.
The \emph{regulator constant}
of~$A$ with respect to $\Theta$ is defined as
$$
\cC_{\Theta}(A) = \prod_j \det\left(\frac{1}{\#U_j}\langle\cdot,\cdot
\rangle|\modStors{(A^{U_j})}\right)^{n_j}\in \bR^\times/(S^\times)^2,
$$
where the $j$-th determinant is computed with respect to any
$S$-basis on $\modStors{(A^{U_j})}$. 
The class in $\bR^\times/(S^\times)^2$ of each determinant is independent of the
choice of basis.
\end{definition}
\begin{theorem}\label{thm:regconst}
The value of $\cC_{\Theta}(A)$ is independent of the pairing $\langle\cdot,\cdot\rangle$
on~$A$.
\end{theorem}
\begin{proof}
See \cite[Theorem 2.17]{tamroot}.
\end{proof}

It follows from Theorem \ref{thm:regconst} that if $S=\Z$, then
$\cC_{\Theta}(A)\in \Q^\times$ is well-defined. Indeed, $(\Z^{\times})^2=\{1\}$,
and the pairing can always be chosen to take values in $\Q$. Similarly,
if $S=\Z_{(p)}$, then $\ord_p(\cC_{\Theta}(A))\in \Z$ is well-defined,
while if $S=\Q$, then $\cC_{\Theta}(A)$ defines a class in $\Q^\times/(\Q^\times)^2$.

\begin{corollary}\label{cor:multiplicative}
Let $A_1$, $A_2$ be $S[G]$-modules, and $\Theta$ a $G$-relation. Then
$\cC_{\Theta}(A_1\oplus A_2) = \cC_{\Theta}(A_1)\cC_{\Theta}(A_2)$.
\end{corollary}
\begin{proof}
We can choose the pairing on $A_2\oplus A_2$ so that the direct summands are
orthogonal to each other, making all the matrices block diagonal.
\end{proof}

\begin{example}\label{ex:regconsttriv}
Let $G$ be a finite group, and let 
$\Theta=\sum_j n_j U_j$
be a $G$-relation. Then $\cC_{\Theta}(\triv) = \prod_j \#U_j^{-n_j}$.
\end{example}

\begin{lemma}\label{lem:degree0}
Let $G$ be a finite group and let $\Theta=\sum_j n_jU_j$ be a $G$-relation.
Then we have $\sum_j n_j=0$.
\end{lemma}
\begin{proof}
  By definition, the virtual representation $\bigoplus_j
  \bC[G/U_j]^{\oplus n_j}$ is zero.
The result follows by taking the inner product with the trivial character of $G$.
\end{proof}

\begin{proposition}\label{prop:orthoreg}
  Let~$G$ be a finite group, and let~$\Theta = \sum_j n_j U_j$ be a~$G$-relation.
  Let~$A$ be a $\Q[G]$-module that has no simple summand in common with any~$\Q[G/U_j]$.
  Then~$\cC_{\Theta}(A)=1$.
\end{proposition}
\begin{proof}
See \cite[Lemma 2.26]{tamroot}.
\end{proof}

Let $G$ be a finite group, and $D$ a subgroup. The operation on $G$-sets
of restricting the action of $G$ to $D$ extends linearly
to a restriction map on the Burnside groups, and induces a restriction map
on relations. With respect to the bases of the Burnside groups given
by transitive $G$-sets, respectively transitive $D$-sets,
the restriction map is given by Mackey's formula:
\begin{align}\label{eq:res}
\Res_{D}^G U = \sum_{x\in U\lquo G / D} D \cap x^{-1}Ux.
\end{align}
We have the following form of Frobenius reciprocity for regulator constants.
\begin{proposition}\label{prop:Frobenius}
Let $G$ be a finite group, let $D$ be a subgroup, let $A$ be an $S[D]$-module,
and let $\Theta$ be a $G$-relation. Then we have
$$
\cC_{\Theta}(\Ind^G_D A) = \cC_{\Res_D^G\Theta}(A).
$$
\end{proposition}
\begin{proof}
See \cite[Proposition 2.45]{tamroot}.
\end{proof}

\section{Torsion homology and regulators in relations}\label{sec:manifolds}
In this section, we investigate the quotients of torsion homology
and of regulators of Riemannian manifolds arising from $G$-relations.
We begin by recalling some basic definitions and fixing the notation. The reader
is referred to \cite{Laplace} for the details.
\begin{definition}
If $V$, $W$ are two abelian groups equipped with $\bR$-valued bilinear forms
$\langle\cdot,\cdot\rangle_V$, respectively $\langle\cdot,\cdot\rangle_W$,
a \emph{similitude from $V$ to $W$ with factor\footnote{The similitude factor is
sometimes defined to be the square root of our convention.} $\lambda\in \bR$} is
a homomorphism $f\colon V\rar W$ such that $\langle f(v_1),f(v_2)\rangle_W =
\lambda\langle v_1,v_2\rangle_V$ for all $v_1$, $v_2\in V$. An \emph{isometry}
is a similitude with factor 1.
\end{definition}

\begin{notation}\label{not:pairings}
Let $X$ be a $d$-dimensional Riemannian manifold, and let $i\ge 0$ be an 
integer.
\begin{enumerate}[leftmargin=*, label={\roman*.}]
\item \emph{Harmonic forms.}
The Laplace--de Rham operator acts on the differential $i$-forms on $X$,
and we let $\cH^i(X)$ denote the space of differential $i$-forms that lie in the
kernel of that operator, the space of harmonic $i$-forms. The Riemannian metric
on $X$ induces a canonical inner product on $\cH^i(X)$, which induces a canonical
isomorphism between $\cH^i(X)$ and the $\bR$-linear dual $\cH^i(X)^{\lor}=\Hom(\cH^i(X),\bR)$.
Explicitly, the inner product on $\cH^i(X)$ is given by
$$
(\omega_1,\omega_2)_X^i = \int_X \omega_1\wedge *\omega_2,
$$
where $\omega_1$, $\omega_2\in \cH^i(X)$, and where $*$ denotes the Hodge
star operator $*\colon \cH^i(X)\rar \cH^{d-i}(X)$ (see \cite[\S 1.2.3]{Laplace}).
We equip~$\cH^i(X)^{\lor}$ with the inner product induced by the one
on~$\cH^i(X)$, i.e. the unique inner product that makes the map
$\omega \mapsto (\cdot, \omega)_X^i$ an isometry.

\item\label{not:reg} \emph{Regulators.}
We have a homomorphism
\[
  h_X^i\colon H_i(X,\Z)\rar \cH^i(X)^{\lor}
\]
given by $h_X^i(\gamma) = (\omega\mapsto \int_\gamma \omega)$.
It follows from the Hodge Theorem (\cite[Theorem 1.45]{Laplace}) and
de~Rham's Theorem that $\ker h_X^i=H_i(X,\Z)_{\tors}$ and that the image of
$h_X^i$ spans~$\cH^i(X)^{\lor}$ over $\bR$.
The \emph{$i$-th regulator} $\Reg_i(X)$ of~$X$ is defined as the covolume
of~$h_X^i(H_i(X,\Z))$ with respect to the inner product on~$\cH^i(X)^{\lor}$.

Let~$\langle\cdot,\cdot\rangle_X^i\colon H_i(X,\Z)\otimes_{\Z}H_i(X,\Z)\rar \bR$
denote the pullback via~$h_X^i$ of the pairing on~$\cH^i(X)^{\lor}$, i.e. the
unique pairing on $H_i(X,\Z)$ that makes~$h_X^i$ an isometry.

\item \emph{Poincar\'e duality.}
By Poincar\'e duality, the map
\[
  D_X\colon H^{d-i}(X,\Z)\rar H_i(X,\Z),
\]
given by cap product with the fundamental class of $X$ is an isomorphism.
It follows that~$\Reg_i(X) = \Reg_{d-i}(X)^{-1}$.

\item\label{rmrk:equivariant} \emph{Actions of automorphism groups.} Let $G$ be a group of automorphisms of $X$.
Then $G$ acts on the left by linear transformations on $H_i(X,\Z)$ via pushforward,
$\gamma\mapsto \sigma_*\gamma$ for $\sigma\in G$ and $\gamma\in H_i(X,\Z)$.
Also, $G$ acts linearly on the right on $\cH^i(X)$ via pullback,
$\omega\mapsto \sigma^*\omega$ for $\sigma\in G$ and $\omega\in \cH^i(X)$,
which induces the dual action on the left on $\cH^i(X)^\lor$. It follows
from the adjointness of pullbacks and pushforwards that the map $h_X^i$
is a homomorphism of left $G$-modules. Moreover, recall that our automorphisms
are assumed to be orientation preserving, so the Poincar\'e
duality map $D_X$ is an isomorphism of $G$-modules.
\end{enumerate}
\end{notation}

\begin{lemma}\label{lem:regdet}
Let~$X$ be a Riemannian manifold.
Then we have $\det\langle \cdot ,\cdot \rangle^i_X=\Reg_i(X)^2$, where the left hand
side is computed with respect to any $\Z$-basis of $\modtors{H_i(X,\Z)}$.
In particular, the pairing $\langle \cdot ,\cdot \rangle^i_X$ is non-degenerate
on~$\modtors{H_i(X,\Z)}$.
\end{lemma}
\begin{proof}
  The claimed equality is the expression of the covolume of a lattice in terms
of the determinant of its Gram matrix.
\end{proof}

\begin{remark}\label{rmrk:PoincareReg}
  Explicitly, let~$\omega_1,\ldots,\omega_r$ be an orthonormal basis of
$\cH^i(X)$, and
  let~$(\gamma_j)$ be a~$\Z$-basis of~$\modtors{H_i(X,\Z)}$. Then we have
  \[
    \Reg_i(X) = \left|\det\int_{\gamma_j}\omega_k\right|\text{,}
  \]
  and for all $\gamma$, $\gamma'\in H_i(X,\Z)$, we have
$\langle \gamma,\gamma'\rangle^i_X =
\sum_{k=1}^r(\int_{\gamma}\omega_k\cdot \int_{\gamma'}\omega_k)$.
\end{remark}



We now prove Proposition \ref{prop:intro1} as a consequence of the following result.

\begin{theorem}\label{thm:noMackeyFunctor}
Let $X\rar Y$ be a finite $G$-covering of Riemannian manifolds,
and let $\Theta=\sum_j n_jU_j - \sum_k n_k'U_k'$ be a $\Z_{(p)}[G]$-relation,
where $n_j$ and $n_k'$ are positive integers. Then for every 
integer $i\ge 0$, we have
$$
\bigoplus_j H_i(X/U_j,\Z)[p^\infty]^{n_j}\cong \bigoplus_k H_i(X/U_k',\Z)[p^\infty]^{n_k'}.
$$
In particular we have~$\ord_p(\#H_i(X/\Theta,\Z)_{\tors})=0$
(see Notation \ref{not:BurnsideFunctions}).
\end{theorem}
\begin{proof}
Let $U$ be any subgroup of $G$. Then the $p^{\infty}$-torsion of $H_i(X/U,\Z)$
is isomorphic to
the torsion subgroup of $H_i(X/U,\Z)\otimes\Z_{(p)}\cong H_i(X/U,\Z_{(p)})$.
By definition of homology with local coefficients (see \cite[\S 3.H]{Hatcher}),
we have $H_i(X/U,\Z_{(p)})\cong H_i(X/G,\Z_{(p)}[G/U])$. Since
homology with local coefficients is additive in direct sums of modules,
the result follows.
\end{proof}

\begin{corollary}
Let $X\rar Y$ and $G$ be as in Theorem \ref{thm:noMackeyFunctor}, let $p$ be a
prime number that does not divide $\#G$, and let $\Theta$ be a $G$-relation. Then
$$
\bigoplus_j H_i(X/U_j,\Z)[p^\infty]^{n_j}\cong \bigoplus_k H_i(X/U_k',\Z)[p^\infty]^{n_k'}.
$$
\end{corollary}
\begin{proof}
This is a direct consequence of Theorem \ref{thm:noMackeyFunctor} and Lemma
\ref{lem:Grelations}.
\end{proof}

The rest of the section is devoted to the behaviour of regulators in $G$-relations.
In particular, we will prove Theorem \ref{thm:intro3}.
\begin{lemma}\label{lem:H0}
Let $X\to Y$ be a finite $G$-covering of Riemannian manifolds of
dimension $d$, and let $\Theta=\sum_j n_jU_j$ be a $G$-relation. Then
$$
\Reg_d(X/\Theta)^2 = \Reg_0(X/\Theta)^{-2} = \cC_{\Theta}(\triv).
$$
\end{lemma}
\begin{proof}
  For all $U\leq G$, we have $\Reg_d(X/U)^2=\Reg_0(X/U)^{-2} = \frac{1}{\#U}\vol(X)$.
It follows that
$$
\Reg_d(X/\Theta)^2 = \Reg_0(X/\Theta)^{-2} = \left(\prod_j
\frac{1}{\#U_j^{n_j}}\right)\cdot \vol(X)^{\left(\sum_j
n_j\right)}=\cC_{\Theta}(\triv),
$$
where the last equality follows from Example \ref{ex:regconsttriv} and Lemma
\ref{lem:degree0}.
\end{proof}


\begin{lemma}\label{lem:simil}
Let $f\colon X\rar Y$ be a finite $G$-covering of
Riemannian manifolds, and let $i\ge 0$ be an integer.
Then the following maps are similitudes with factor $\#G$:

\begin{enumerate}
\item\label{item:prepnotisom} the pullback map $f^*\colon \cH^i(Y)\rar \cH^i(X)^{G}$;
\item\label{item:notisom}
its dual $(f^*)^{\lor} \colon (\cH^i(X)^{G})^{\lor} \to \cH^i(Y)^{\lor}$;
\item\label{item:descentpairing}
the pushforward map
$
  f_*\colon H_i(X,\Z)^G \to H_i(Y,\Z);
$
\item\label{item:ascentpairing} any map $g\colon H_i(Y,\Z)\rar H_i(X,\Z)^G$ satisfying
$f_*g = \# G\cdot \mathrm{id}_{H_i(Y,\Z)}$.
\end{enumerate}
\end{lemma}
\begin{proof}
\begin{enumerate}[leftmargin=*]
\item Let $[X]$ and $[Y]$ denote the fundamental class on $X$, respectively on $Y$.
Then $f_*[X] = \#G\cdot [Y]$, so for all~$\omega_1,\omega_2\in\cH^i(Y)$, we have
\begin{align*}
(f^*\omega_1,f^*\omega_2)_X^i & = \int_{X} f^*\omega_1\wedge f^{*} {*\omega_2}
 = \int_{X} f^*(\omega_1\wedge *\omega_2)\\
& = \#G\cdot \int_{Y}\omega_1\wedge *\omega_2
= \#G\cdot (\omega_1,\omega_2)_Y^i.
\end{align*}
\item The assertion immediately follows from part (\ref{item:prepnotisom}),
from the fact that $f^*$ is an isomorphism, and from the definition of the
inner product on the dual space, see Notation \ref{not:pairings}.
\item
By Notation \ref{not:pairings} \ref{rmrk:equivariant}, the map $h_X^i$ sends
$H_i(X,\Z)^G$ to $(\cH^i(X)^G)^{\lor}$,
and it follows from the adjointness of pullbacks and pushforwards that the diagram
  \[
  \xymatrix{
    H_i(X,\Z)^G \ar[d]^{h_X^i} \ar[r]^(.56){f_*} & H_i(Y,\Z) \ar[d]^{h_Y^i} \\
    (\cH^i(X)^G)^{\lor} \ar[r]^(.6){(f^*)^{\lor}} & \cH^i(Y)^{\lor}
  }
  \]
  is commutative.
  Note that $(\cH^i(X)^G)^{\lor}$ is canonically
isomorphic, as an inner product space, to the $G$-coinvariants
$(\cH^i(X)^{\lor})_G$, which in turn is canonically isomorphic to
$(\cH^i(X)^{\lor})^G$. This shows that
$h_X^i\colon H_i(X,\Z)^G \to (\cH^i(X)^G)^{\lor}$ is an isometry.
  By definition, $h_Y^i$ is also an isometry,
  and by part (\ref{item:notisom}), $(f^*)^{\lor}$ is a similitude
  with factor~$\# G$. It follows that $f_*$ is also a similitude with factor~$\# G$.
\item By part~(\ref{item:descentpairing}), the map $f_*$ is a similitude with
factor~$\# G$, and the claim immediately follows.
\end{enumerate}\vspace{-1em}
\end{proof}

\begin{lemma}\label{lem:exconvenient}
  Let~$X\to Y$ be a finite~$G$-covering of Riemannian manifolds, and let~$i\ge 0$ be
  an integer. Then the composition
$$
g_i\colon
H_i(Y,\Z)\stackrel{D_Y^{-1}}{\longrightarrow}H^{d-i}(Y,\Z)\stackrel{f^*}{\longrightarrow}H^{d-i}(X,\Z)^G
\stackrel{D_X}{\longrightarrow} H_i(X,\Z)^G
$$
satisfies $f_*g_i = \# G\cdot \mathrm{id}_{H_i(Y,\Z)}$.
\end{lemma}
\begin{proof}
  The pushforward of the fundamental class of $X$ is equal to $\#G$ times
  the fundamental class of $Y$. It therefore follows from the naturality of the cap
  product that $\# G\cdot D_Y = f_* D_X f^*$, which proves the result.
\end{proof}

\begin{lemma}\label{lem:index}
  Let~$X\to Y$ be a finite $G$-covering of Riemannian manifolds, and let~$i\ge
  0$ be an
  integer. Let~$g\colon H_i(Y,\Z)\rar H_i(X,\Z)^G$ be any map
  satisfying $f_*g = \# G\cdot \mathrm{id}_{H_i(Y,\Z)}$, and let
$$
\lambda_i(X,G,g)=\left(H_i(X,\Z)^{G}\colon g(H_i(Y,\Z)) +
H_i(X,\Z)^{G}_{\tors}\right).
$$
Then~$\lambda_i(X,G,g)$ is finite and is independent of $g$.
\end{lemma}
\begin{proof}The quantity~$\lambda_i(X,G,g)$ is finite since
$g\otimes \Q\colon H_i(Y,\Q)\rar H_i(X,\Q)^{G}$ is an isomorphism.
If~$g'$ is another map satisfying the conditions of the lemma,
then~$f_*(g-g') = 0$. It follows from this, and from the fact
that~$f_*\otimes\Q$ is
an isomorphism, that~$\Im(g-g')\subset \ker f_* \subset
H_i(X,\Z)^G_{\tors}$.
Therefore
$$
g(H_i(Y,\Z)) + H_i(X,\Z)^{G}_{\tors}=
g'(H_i(Y,\Z)) + H_i(X,\Z)^{G}_{\tors},
$$
which proves the lemma.
\end{proof}

Given a finite $G$-covering $X\rar Y$ and $i$ as in Lemma \ref{lem:index},
we define $\lambda_i(X,G)=\lambda_i(X,G,g)$ for any $g$ satisfying
the hypotheses of the lemma.

\begin{theorem}\label{thm:regcomparison}
Let $X\to Y$ be a~$G$-covering of Riemannian manifolds, let~$\Theta=\sum_j
n_jU_j$ be a $G$-relation (see Definition~\ref{defn:relation}), and let~$i\ge 0$
be an integer. Then we have
$$
\Reg_i(X/\Theta)^2 = \cC_{\Theta}(H_i(X,\Z))\cdot\lambda_i(X,\Theta)^2.
$$
\end{theorem}
\begin{proof}
By Lemma \ref{lem:regdet}, the pairing $\langle\cdot,\cdot\rangle_X^i$ on
$H_i(X,\Z)$ is non-degenerate on the quotient modulo torsion. Moreover,
by Lemma \ref{lem:simil} (\ref{item:descentpairing}) applied to
$f=\sigma$ for each $\sigma\in G$, this pairing is $G$-invariant.
For any $U_j\leq G$, let~$g$ be any map satisfying the conditions of Lemma
\ref{lem:index} for the covering~$X\to X/U_j$. Then we have
\begin{align*}
  \MoveEqLeft{\det\left(\frac{1}{\#U_j}\langle\cdot,\cdot\rangle_X^i \mid \modtors{H_i(X,\Z)^{U_j}}\right)}\\
& =
  \det\left(\frac{1}{\#U_j}\langle\cdot,\cdot\rangle_X^i \mid
  \modtors{g(H_i(X/U_j,\Z))}\right) / \lambda_i(X,U_j)^2\\
  & = \det\left(\langle\cdot,\cdot\rangle_{X/U_j}^i \mid \modtors{H_i(X/U_j,\Z)}\right) / \lambda_i(X,U_j)^2\\
& = \Reg_i(X/U_j)^2 / \lambda_i(X,U_j)^2,
\end{align*}
where the second equality follows from Lemma \ref{lem:simil}(\ref{item:ascentpairing}),
and the last one from Lemma \ref{lem:regdet}.
This proves the theorem, by taking the product over $j$ corresponding to $\Theta$.
\end{proof}

\begin{corollary}\label{cor:modsquares}
Let $X\to Y$ be a~$G$-covering of Riemannian manifolds, let~$\Theta=\sum_j
n_jU_j$ be a $G$-relation, and let~$i\ge 0$ be an 
integer. Then we have
$\Reg_i(X/\Theta)^2\equiv \cC_{\Theta}(H_i(X,\Q)) \pmod{(\Q^\times)^2}$.
\end{corollary}

Let $U$ be a group. The abelianisation of $U$ is the quotient
$U^{\ab}=U/[U,U]$, where $[U,U]=\langle uvu^{-1}v^{-1}\colon u, v\in U\rangle$, i.e.
the maximal abelian quotient of~$U$.

\begin{corollary}\label{cor:regcomparison}
Let $f\colon X\rar Y$ be a finite $G$-covering of Riemannian $d$-manifolds,
let $\Theta=\sum_j n_jU_j$ be a $G$-relation,
let $p$ be a prime, and let~$i\ge 0$ be an integer.
Assume that one of the following holds:
\begin{enumerate}
  \item $i=d-1$ and for each subgroup $U_j\leq G$, 
    the order of $U_j^{\ab}$ is not divisible by~$p$;
  \item for each subgroup $U_j\leq G$, 
    the order of $U_j$ is not divisible by~$p$.
\end{enumerate}
Then $\ord_p(\Reg_i(X/\Theta)^2) = \ord_p(\cC_{\Theta}(H_i(X,\Z))) = \ord_p(\cC_{\Theta}(H_i(X,\Z_{(p)})))$.

In particular, for all~$i\ge 0$ the rational number~$\Reg_i(X/\Theta)^2$ is a product of the primes dividing the
order of~$G$.
\end{corollary}
\begin{proof}
  Let $U$ be any subgroup of $G$. Let~$g_i=D_Xf^*D_{X/U}^{-1}$ be the map of
Lemma~\ref{lem:exconvenient}.
In case~(1), the 5-term exact sequence coming from the Cartan--Leray spectral sequence
in cohomology (see \cite[Ch. VII, \S 7]{Brown}) is
$$
0\rar H^1(U,\Z)\rar H^1(X/U,\Z)\stackrel{f^*}{\rar} H^1(X,\Z)^U\rar H^2(U,\Z).
$$
By the universal coefficient theorem, we have $H^2(U,\Z)\cong H_1(U,\Z) =
U^{\ab}$, which has order coprime to $p$ by assumption.
In case~(2), the group~$U$ has
cohomological dimension~$0$ over~$\Z_{(p)}$, since~$p$ does not divide its order,
so by the Cartan--Leray spectral
sequence we have an isomorphism
\[
  H^{d-i}(X/U,\Z_{(p)})\stackrel{f^*}{\rar} H^{d-i}(X,\Z_{(p)})^U.
\]

In both cases, it follows that $f^*\colon H^{d-i}(X/U,\Z)\rar H^{d-i}(X,\Z)^U$
has cokernel of order coprime to $p$, and since the Poincar\'e duality maps
$D_X$ and~$D_{X/U}$ are isomorphisms, the cokernel of
$g_i\colon H_{i}(X/U,\Z)\rar H_{i}(X,\Z)^U$ is also of order coprime to $p$.
The first equality therefore follows from Theorem~\ref{thm:regcomparison}.
For the second equality, note that firstly,
$H_i(X,\Z_{(p)})\cong H_i(X,\Z) \otimes_{\Z} \Z_{(p)}$, and secondly that for an
arbitrary $\Z[G]$-module $A$, we have $\ord_p(\cC_{\Theta}(A)) = 
\ord_p(\cC_{\Theta}(A\otimes_{\Z}\Z_{(p)}))$.
\end{proof}

\begin{remark}\label{rmrk:noclue}
We do not know whether the regulator quotient is a purely representation
theoretic invariant in full generality without the restrictive assumptions of
Corollary \ref{cor:regcomparison}, that is whether it only depends
on the $\Z[G]$-module structure of $H_i(X,\Z)$.
This seems to us to be an interesting question.
\end{remark}

\section{Some regulator constant calculations}\label{sec:regconst}
In this section, we compute the regulator constants of rational representations
with respect to the relations of Examples \ref{ex:2-Gassman} and \ref{ex:p-Gassman}.

First, let $G=\GL_2(\F_p)$.
Denote by~$B$ the subgroup of upper-triangular matrices in~$G$,
let~$w = \left(\begin{smallmatrix}0 & 1 \\ 1 & 0\end{smallmatrix}\right)$,
and let~$U_1$, respectively $U_2$ be the subgroup of matrices in~$B$ with square
lower right, respectively upper left entry. Then~$\Theta = U_1 - U_2$ is the
$G$-relation of Example \ref{ex:p-Gassman}.
In this section, we compute the regulator constants of the
rational irreducible representations of $G$ with respect to $\Theta$.

We begin by recalling the classification of complex representations of~$G$.
For any 1-dimensional representations $\mu_1$, $\mu_2$ of~$\F_p^\times$,
let~$r_{\mu_1,\mu_2}$ be the $1$-dimensional representation of~$B$
defined by $r_{\mu_1,\mu_2}\colon
\left(\begin{smallmatrix}a & b \\ 0 & d\end{smallmatrix}\right) \mapsto \mu_1(a)\mu_2(d)$.
\begin{proposition}\label{prop:classif}
  The following is a classification of the complex irreducible representations
  of~$G$:
  \begin{itemize}
    \item $1$-dimensional representations: they are all of the form~$\mu\circ\det$,
      where~$\mu$ is an irreducible character of~$\F_p^\times$.
    \item Irreducible principal series, of dimension~$p+1$: they are
      the inductions~$\rho(\mu_1,\mu_2) = \Ind_B^G r_{\mu_1,\mu_2}$,
      where~$\mu_1\neq\mu_2$; we have $\rho(\mu_1,\mu_2)\cong
      \rho(\mu_1',\mu_2')$ if and only if $\{\mu_1,\mu_2\}=\{\mu_1',\mu_2'\}$.
    \item Special representations, of dimension~$p$: they are
      the~$p$-dimensional summands~$\sigma(\mu)$ of~$\Ind_B^G r_{\mu,\mu}\cong
      \sigma(\mu)\oplus (\mu\circ\det)$.
    \item Cuspidal representations, of dimension~$p-1$: they are the irreducible
      representations of~$G$ that are not summands of any~$\Ind_B^G r_{\mu_1,\mu_2}$.
  \end{itemize}
\end{proposition}
\begin{proof}
See \cite{PShapiro}.
\end{proof}
Let $\St$ denote the complex irreducible representation
$\sigma(\triv)$, the Steinberg representation of $G$, and let
$I$ denote the complex irreducible representation~$\rho(\chi,\triv)$,
where $\chi$ is the unique character of $\F_p^\times$ of order 2.
Note that both $\St$ and $I$ are realisable over $\Q$, so in fact give rise to
irreducible $\Q[G]$-representations.

\begin{proposition}\label{prop:regconstpGassman}
Let $\rho$ be an irreducible $\Q[G]$-representation. Then\newline
$\cC_{\Theta}(\rho)\equiv p \mod{(\Q^\times)^2}$ if $\rho\cong I$, and
$\cC_{\Theta}(\rho)\equiv 1 \mod{(\Q^\times)^2}$ otherwise.
\end{proposition}
\begin{proof}
We have $\C[G/U_1]\cong \C[G/U_2] \cong \triv\oplus \St\oplus I$. So if
$\rho$ is an irreducible $\Q[G]$-representation that is not isomorphic to
any of these three direct summands, then $\cC_{\Theta}(\rho)=1$ by
Proposition \ref{prop:orthoreg}. Moreover, we have $\cC_{\Theta}(\triv) = \#U_2/\#U_1=1$
by Example \ref{ex:regconsttriv}.

To compute $\cC_{\Theta}(\St)$, observe that by Corollary \ref{cor:multiplicative}
and by Proposition \ref{prop:Frobenius},
we have
\begin{align}\label{eq:Steinberg}
\cC_{\Theta}(\St)=\cC_{\Theta}(\St)\cC_{\Theta}(\triv) = \cC_{\Theta}(\Ind_B^G\triv) =
\cC_{\Res_B^G\Theta}(\triv).
\end{align}
Let $T_j=U_j\cap T$ for $j=1$, $2$. An elementary
calculation (or the Bruhat decomposition) shows that for $j=1$ and $2$,
$U_j\lquo G / B$ has cardinality 2,
with double coset representatives $1$ and~$w$. Moreover, we have
$w^{-1}U_1w\cap B = T_2$, and $w^{-1}U_2w\cap B = T_1$.
It follows from equation (\ref{eq:res}) that
$
\Res_B^G\Theta = U_1 + T_2 - U_2 - T_1.
$
We conclude from equation (\ref{eq:Steinberg}) and Example \ref{ex:regconsttriv}
that
$$
\cC_{\Theta}(\St) = \frac{\#U_2\cdot \#T_1}{\#U_1\cdot \#T_2} = 1 \mod{(\Q^\times)^2}.
$$

Similarly, we have
$$
\cC_{\Theta}(I) = \cC_{\Theta}(\Ind_B^Gr_{\chi,\triv}) =
\cC_{\Res_B^G\Theta}(r_{\chi,\triv}).
$$
Since $r_{\chi,\triv}$ is 1-dimensional, the contribution from each subgroup $U$
in $\Res_B^G\Theta$ is 1 if $(r_{\chi,\triv})^U$ is zero, and is $(\#U)^{-1}$
otherwise. It follows that $\cC_{\Theta}(I) \equiv \#U_2/\#T_2 \equiv p \mod{(\Q^\times)^2}$.
\end{proof}
\begin{corollary}\label{cor:impliesconj}
Let $p$ be an odd prime, let $G$, $B$, $U_1$, and $U_2$ be as above,
and let $X\rar Y$ be a $G$-covering of closed hyperbolic $3$-manifolds.
Then the following are congruent modulo 2:
\begin{enumerate}
  \item the $p$-adic valuation
    $\ord_p\left(\frac{\#H_1(X/U_1,\Z)_{\tors}}{\#H_1(X/U_2,\Z)_{\tors}}\right)$;
\item the multiplicity in $H_1(X,\C)$ of the irreducible $\C[G]$-representation
$I$, as defined after Proposition \ref{prop:classif};
\item the difference of the Betti numbers $b_1(X/U_1) - b_1(X/B)$.
\end{enumerate}
\end{corollary}
\begin{proof}
First, we show the equality between (1) and (2).
For closed hyperbolic $3$-manifolds, all homology groups with integral coefficients
except possibly the first are torsion-free. It therefore follows from equation
(\ref{eq:CheegerMuller}), Remark~\ref{rmrk:PoincareReg}, Lemma \ref{lem:H0}, and
Example \ref{ex:regconsttriv}, that
$$
\frac{\#H_1(X/U_1,\Z)_{\tors}}{\#H_1(X/U_2,\Z)_{\tors}} = \frac{\Reg_0(X/U_2)^2\Reg_1(X/U_1)^2}
{\Reg_0(X/U_1)^2\Reg_1(X/U_2)^2}=\frac{\Reg_1(X/U_1)^2}{\Reg_1(X/U_2)^2}.
$$
By Corollary \ref{cor:modsquares}, the $p$-adic valuation of the right hand
side is congruent to $\ord_p\cC_{\Theta}(H_1(X,\Z))\equiv \ord_p\cC_{\Theta}(H_1(X,\Q))\pmod{2}$,
where $\Theta=U_1-U_2$. Recall that the right hand side of that last congruence is
only well-defined modulo 2. The equality of (1) and (2) therefore follows from
Corollary \ref{cor:multiplicative} and Proposition \ref{prop:regconstpGassman}.

We now prove the equality between (2) and (3). Below, we will identify
representations of $G$ with their characters.
Let $\chi$ be an arbitrary rational representation of $G$.
We have
$I=\Ind_{U_1}^G\triv - \Ind_B^G\triv$, so the multiplicity of $I$ in $\chi$
is equal to
\begin{eqnarray*}
\langle I,\chi\rangle_G & = & \langle\Ind_{U_1}^G\triv -
\Ind_B^G\triv,\chi\rangle_G = \langle\triv,\Res_{U_1}\chi\rangle_{U_1} - 
\langle\triv,\Res_{B}\chi\rangle_{B}\\
& = &\dim_{\Q}\chi^{U_1} - \dim_{\Q}\chi^{B},
\end{eqnarray*}
where $\langle\cdot,\cdot\rangle_U$ denotes inner products of characters
of $U$.
The equality between (2) and (3) follows from this calculation,
with $\chi=H_1(X,\Q)$.
\end{proof}
As in Example \ref{ex:p-Gassman}, let
$N=\{\left(\begin{smallmatrix}a&0\\0&a\end{smallmatrix}\right)\colon a\in (\F_p^\times)^2\}$,
let $\bar{G}=G/N$, $\bar{U}_j=U_j/N$ for $j=1$, $2$, and let $\bar{\Theta}=\bar{U}_1-\bar{U}_2$,
which is a $\bar{G}$-relation. There is a canonical bijection between
irreducible $\Q[G]$-representations with kernel containing $N$ and irreducible
$\Q[\bar{G}]$-representations, given by projecting to the quotient. For any
irreducible $\Q[G]$-representation $\rho$ with kernel containing $N$, let $\bar{\rho}$
denote its image under this bijection.

\begin{corollary}
Let $\bar{\rho}$ be an irreducible $\Q[\bar{G}]$-representation. Then\newline
$\cC_{\bar{\Theta}}(\bar{\rho})\equiv p \mod{(\Q^\times)^2}$ if $\bar{\rho}\cong \bar{I}$,
and $\cC_{\bar{\Theta}}(\bar{\rho})\equiv 1 \mod{(\Q^\times)^2}$ otherwise.
\end{corollary}
\begin{proof}
  By \cite[Proposition 2.45]{tamroot}, we have $\cC_{\bar{\Theta}}(\bar{\rho})=
  \cC_{\Theta}(\rho)$. The result therefore immediately follows from
  Proposition \ref{prop:regconstpGassman}.
\end{proof}

Now let $G$ be the group of affine linear transformations 
$T_{a,b}\colon x\mapsto ax+b$ of $\Z/8\Z$, where $a\in (\Z/8\Z)^{\times}$
and $b\in \Z/8\Z$, let $U_1=\langle T_{a,0}\colon a \in (\Z/8\Z)^\times\rangle$
and $U_2= \langle T_{-1,0}, T_{3,4}\rangle$, as in Example \ref{ex:2-Gassman}.
Let $\chi$ be an irreducible faithful representation of
$N=\langle T_{1,b}\colon b \in \Z/8\Z\rangle$. Then $I = \Ind_N^G\chi$ is a
4-dimensional irreducible $\C[G]$-representation, and is realisable over $\Q$,
so it gives rise to an irreducible $\Q[G]$-representation. Its isomorphism class
does not depend on the choice of the representation $\chi$.
All the other complex irreducible representations of $G$ are also realisable
over $\Q$, and are 1- and 2-dimensional.
An explicit calculation, which we omit, gives the regulator constants
of all the irreducible $\Q[G]$-representations with respect to $\Theta=U_1-U_2$
as follows.
\begin{proposition}\label{prop:2regcst}
Let $\rho$ be an irreducible $\Q[G]$-representation. Then\newline
$\cC_{\Theta}(\rho)\equiv 2\mod{(\Q^\times)^2}$ if $\rho\cong I$, and
$\cC_{\Theta}(\rho)\equiv 1\mod{(\Q^\times)^2}$ otherwise.
\end{proposition}

\section{Computations}\label{sec:calc}

In this section, we present the computational methods used to prove
Proposition~\ref{prop:intro2}. The computations were run on the Warwick number
theory cluster. We have made the resulting data available at
\url{https://arxiv.org/abs/1601.06821}.

For each prime $p$, we look for explicit pairs of manifolds $M_1$ and $M_2$
as in Proposition \ref{prop:intro2}. The pairs we construct all arise as intermediate
coverings of $G$-coverings $X\rar Y$ of closed hyperbolic $3$-manifolds,
for suitable finite groups~$G$. We compute many such coverings, and explicitly
compute the homology groups of suitable intermediate covers. There are several
choices involved, which we now explain in detail.

\subsection{The choice of $Y$.}\label{sec:Y}
We only compute with arithmetic hyperbolic $3$-manifolds.
We will recall the definitions, assuming basic facts
and terminology on quaternion algebras; a standard reference
is~\cite{MaclachlanReid}.

Let $\hyp^3$ denote the hyperbolic 3-space.
We identify the group of orientation preserving
isometries of $\hyp^3$ with $\PSL_2(\C)$ via
the Poincar\'e extension (\cite[p. 48]{MaclachlanReid}).
For any order $\order$ in a quaternion algebra, let~$\order^1$ denote the
group of elements in $\order$ of reduced norm~$1$. 
A number field is called \emph{almost totally real (ATR)} if it has exactly one
complex place. A quaternion algebra over an ATR number field~$F$ is called
\emph{Kleinian} if it ramifies at every real place of~$F$. For any ring $R$,
let $\mat_2(R)$ denote the algebra of $2\times 2$ matrices over~$R$.

\begin{definition}
An \emph{arithmetic hyperbolic $3$-manifold} is a manifold commensurable with
a quotient~$\hyp^3/\Gamma(\order)$, where $\Gamma(\order)$ arises as follows.
Let~$B/F$ be a Kleinian quaternion algebra over an ATR number field, and let
$\iota\colon B\otimes_{F}\C \cong \mat_2(\C)$ be an isomorphism induced by a
complex embedding of~$F$. Then define
$
\Gamma(\order) = \iota(\order^1)/\{\pm1\}\subset\PSL_2(\C).
$
\end{definition}

An arithmetic hyperbolic $3$-manifold has finite volume and is compact if
and only if~$B$ is a division algebra -- see \cite[Theorem 8.2.2]{MaclachlanReid}.

Let $B/F$ be a quaternion algebra over a number field,
let~$\Z_F$ denote the ring of integers in~$F$, and let $\order_{\max}$
be a maximal order in $B$. Let~$\ideal$
be an ideal of~$\Z_F$ coprime to the discriminant of~$B$, so
that~$\order_{\max}\otimes_{\Z_F}\Z_F/\ideal\cong \mat_2(\Z_F/\ideal)$, and
let~$\order_0(\ideal)\subset \order_{\max}$ be the preimage of the subring of
upper-triangular matrices
under such an isomorphism. Then $\order_0(\ideal)$ is an order in~$B$. 

\begin{lemma}\label{lem:torfree}
  Let~$F$ be a number field, and let $B$ be a quaternion algebra over~$F$
  that is a division algebra. Let~$\cS$ be the set of primes~$q$ such
  that~$F(\zeta_{2q})$ is isomorphic to a quadratic extension of $F$ contained
  in~$B$, where $\zeta_{2q}$ denotes a primitive $(2q)$-th root of unity in an
  algebraic closure of $F$. For each $q\in \cS$, and for each primitive $(2q)$-th
  root of unity $\zeta$, choose a prime ideal $\prm_{\zeta}$ of
  $\Z_F$ such that the minimal polynomial of $\zeta$ is irreducible modulo
  $\prm_{\zeta}$. Let~$\cP$ be the set of such $\prm_{\zeta}$. Then:
  \begin{enumerate}
  \item the set $\cS$ is finite;
  \item for every ideal $\ideal$ of $\Z_F$ divisible by all $\prm \in \cP$,
  the group $\order_0(\ideal)^1/\{\pm1\}$ is torsion-free.
  \end{enumerate}
\end{lemma}
\begin{proof}
  If~$q\in \cS$ is odd, then~$q-1 = [\Q(\zeta_{2q})\colon\Q]\le
  [F(\zeta_{2q})\colon\Q] =
  2[F\colon\Q]$, so that~$q\le 2[F\colon\Q]+1$. This proves (1).

  To prove (2), assume for a contradiction that there exists an
  element~$x\in\order_0(\ideal)^1$ such that~$x$ has prime order~$q$
  modulo~$\{\pm 1\}$. Replacing~$x$ by~$-x$ if necessary, we may assume that~$x$
  has order~$2q$: if $q$ is odd then either $x$
  or~$-x$ has order~$2q$; if $q=2$, then the only elements of order~$2$ are~$\pm 1$,
  since~$B$ is a division algebra, so that $x$ has order~$4$.
  In particular, $F(x)$ is a subfield of~$B$ isomorphic to $F(\zeta_{2q})$.
  We claim that this is a quadratic extension of~$F$, i.e. that~$q$ belongs to
  $\cS$: if~$F(\zeta_{2q})$ is not quadratic,
  then~$x\in F$, but~$x^2 = \nrd(x) = 1$, and~$x$ does not have order~$2q$.

  Let~$P$ be the minimal polynomial of~$x$, and let~$\prm\in \cP$ be the
  corresponding prime ideal. By definition, the image of~$x$ in $\mat_2(\Z_F/\ideal)$
  upper-triangular. Since $\prm\mid \ideal$, this implies that~$P$ is not
  irreducible modulo~$\prm$: a contradiction.
\end{proof}

In our search, we extract ATR number fields of degree up to 6 and discriminants
with absolute value up to $10^6$ from the PARI number fields database~\cite{paridb,lmfdb}.
Then, using the algorithms from~\cite{quatalg} implemented in the computer
algebra system Magma~\cite{magma},
we compute 30,000 Kleinian quaternion algebras $B$ whose 
maximal orders~$\order_{\max}\subset B$ are such that~$\Gamma(\order_{\max})$
has covolume at most~$40$. Using the sufficient condition of Lemma
\ref{lem:torfree}, we find many orders $\order_0(\ideal)$ such that
$\Gamma=\Gamma(\order_0(\ideal))$ is torsion-free, and therefore acts freely
on $\hyp^3$, and hence we produce many arithmetic hyperbolic
$3$-manifolds~$Y = \hyp^3/\Gamma$.

\begin{remark}
The usual way of proving that arithmetic groups have a torsion-free subgroup of
finite index is to consider principal congruence subgroups, which have 
large index. Lemma \ref{lem:torfree} allows us to
use groups of the form~$\Gamma(\order_0(\ideal))$, which have smaller covolume,
reducing the cost of the computation.
\end{remark}

\subsection{The choice of the covering.}\label{sec:X}
Theorem \ref{thm:noMackeyFunctor} implies that in order to produce examples
as in Proposition \ref{prop:intro2} using Sunada's method, we need to look
for $G$-coverings $X\rar Y$, where $G$ is a finite group that admits
a $G$-relation $\Theta=U_1-U_2$ that is not a $\Z_{(p)}[G]$-relation. For $p=2$,
we take $G$ and $\Theta$ as in Example \ref{ex:2-Gassman}, and for $p$ odd,
we use the relation of Example \ref{ex:p-Gassman}.

In order to find many such $G$-coverings of every manifold $Y=\hyp^3/\Gamma$
obtained in Section \ref{sec:Y}, we first compute a finite presentation
of~$\Gamma$, using the algorithms of~\cite{klngps}.
Since~$\Gamma \cong \pi_1(Y)$, $G$-coverings of~$Y$ correspond to surjective
homomorphisms~$\Gamma\to G$. Using the presentation of~$\Gamma$, we
enumerate surjective homomorphisms~$\Gamma \to G$ up to conjugacy,
using the methods described in \cite[\S 9.1]{Holt}.
The complexity of the enumeration depends heavily
on~$\# G$, so we use the following improvement.

Let~$F_n$ be the free group on~$n$ generators. Recall that for any group~$Z$
there is a canonical bijection~$\Hom(F_n,Z)\cong Z^n$. Moreover, if~$Z$ is an
abelian group, then this bijection is an isomorphism of abelian groups. We
will tacitly use this identification in the next result.
\begin{prop}
  Let~$\Gamma$ be a group with a finite presentation
  \[
    1\to R \to F_n \to \Gamma \to 1\text{,}
  \]
  where~$R$ is the normal closure of the subgroup of $F_n$ generated by~$r$
  elements~$w_1,\dots,w_r$.
  Let~$G$ be a group, let~$Z$ be a subgroup of the center of~$G$, and
  set~$\bar{G} = G/Z$.
  Let~$\bar{h}\colon \Gamma \to \bar{G}$ be a homomorphism, and let~$\tilde{h}\colon F_n
  \to G$ be an arbitrary lift of~$\bar{h}\colon F_n \to \bar{G}$.
  Let~$x = (\tilde{h}(w_i)^{-1})_{i=1}^r\in G^r$, and let~$\phi\colon Z^n\to Z^r$ denote the
  linear map induced by evaluation on the~$(w_i)_{i=1}^r$. Then:
  \begin{enumerate}
  \item we have $x\in Z^r$; and
  \item the lifts~$h\colon \Gamma \to G$ of~$\bar{h}$ are exactly the
  homomorphisms~$\tilde{h}\cdot f$, where~$f$ ranges over~$\phi^{-1}(x)$.
  \end{enumerate}
\end{prop}
\begin{proof}
  \begin{enumerate}[leftmargin=*]
  \item Since~$\tilde{h}$ is a lift of~$\bar{h}\colon \Gamma \to \bar{G}$, 
  we have~$\tilde{h}(R)\subset Z$, and~$x\in Z^r$.
  \item Since~$Z$ is central, every lift of~$\bar{h}\colon F_n \to \bar{G}$ is of the
  form~$\tilde{h}\cdot f$, where~$f\in\Hom(F_n,Z)=Z^n$. Such a lift descends to a
  homomorphism~$\Gamma \to G$ if and only if it vanishes on the generators
  of~$R$. For every such generator~$w$, we have~$(\tilde{h}\cdot f)(w)=1$ if and
  only if~$f(w) = \tilde{h}(w)^{-1}$,
  so that the homomorphism descends if and only if~$\phi(f)=x$.
  \end{enumerate}\vspace{-1em}
\end{proof}
Applying this proposition to~$G=\GL_2(\F_p)$,~$\bar{G} = \PGL_2(\F_p)$
and~$Z=\F_p^\times$, we reduce the enumeration of~$\Hom(\Gamma,G)$ to:
\begin{itemize}
\item enumerating~$\Hom(\Gamma,\bar{G})$,
\item computing inverse images under linear maps~$(\Z/N\Z)^n\to(\Z/N\Z)^r$
with $N=p-1$, which can be done efficiently by linear algebra.
\end{itemize}
Since the enumeration algorithm is expensive, another simple improvement 
was useful: if there exists a surjection~$\Gamma\to G$, then~$\Gamma^{ab}$
surjects onto~$G^{ab}$. Checking this condition before the enumeration is fast
and rules out many groups.

\subsection{Homology computation.}\label{sec:homcomp}
Given a surjective homomorphism~$\Gamma\to G$, we can compute the
homology of the isospectral manifolds as follows. For every~$U\leq G$, we have
\[
  H_1(X/U,\Z)\cong H_1(Y,\Z[G/U])\cong H_1(\Gamma,\Z[G/U])\text{,}
\]
and we can compute the first homology group by linear algebra from a
presentation of the group~$\Gamma$. 
Recall that we want to find examples of $G$-coverings~$X\to Y$ such that
$\# H_1(X/\Theta,\Z)[p^\infty] \neq 1$.

Since linear algebra over~$\Z$ is costly, we
used the following strategy: we first
compute~$H_1(X/U_j,\Z)\otimes_{\Z}\F_p\cong H_1(X/U_j,\F_p)$ for $j=1$, $2$,
which only uses linear algebra over~$\F_p$,
and we only compute the
homology over~$\Z$ if~$\dim_{\F_p} H_1(X/\Theta,\F_p)\neq 0$.  
That condition by itself already implies that $H_1(X/U_1,\Z)[p^\infty]\not\cong
H_1(X/U_2,\Z)[p^\infty]$.
In practice, when~$p\geq 5$, this condition often yields
examples in which $H_1(X/U_j,\Z)[p^\infty]$ is actually trivial for one~$j$,
and isomorphic to~$\Z/p\Z$ for the other one.


By computing explicit examples, we obtain Proposition~\ref{prop:intro2}, which we
restate here for convenience.

\begin{prop}\label{intro2bis}
Let $p\leq \maxprime$ be a prime number. Then there exist strongly isospectral
closed hyperbolic $3$-manifolds $M_1$ and $M_2$ such that
\[
  \#H_1(M_1,\Z)[p^\infty]\neq \#H_1(M_2,\Z)[p^\infty].
\]
Moreover,
if~$2<p\leq \maxprime$, then there exist $3$-manifolds $M_1$
and $M_2$ as above with
$$
\#H_1(M_1,\Z)[p^\infty]=1,\;\;\text{ and }\;\;\#H_1(M_2,\Z)[p^\infty]=p.
$$
\end{prop}


\section{Examples}\label{sec:examples}
In this section, we give several numerical examples, illustrating
some features of the connection between the torsion on the homology of isospectral
manifolds, and the $G$-module structure of the integral homology of their
covering manifold.
\begin{example}\label{ex:1}
Let $G=\GL_2(\F_{37})$, let $U_1$, $U_2$ be as in Example \ref{ex:p-Gassman}.
Using the notation of Section \ref{sec:Y}, let $F$ be the ATR quartic number
field generated by a root~$t$ of $x^4-2x-1$. Consider the Kleinian quaternion
algebra
$$
B=\langle 1,i,j,k \mid i^2=j^2=k^2=ijk=-1\rangle_F
$$
over $F$. Let $\ideal$ be the principal ideal of $\Z_F$ generated by
$t^2-t-2$, which has norm~$22$, and let $\Gamma_1=\Gamma(\order_0(\ideal))$,
which is well-defined up to conjugacy in~$\PSL_2(\C)$. A fundamental domain
for~$\Gamma_1$ acting on the ball model of~$\hyp^3$ is displayed in
Figure~\ref{pic37}. The group~$\Gamma_1$
has a presentation by generators and relations as follows:
\begin{eqnarray*}
\Gamma_1 \cong\langle a,b,c,d & | & b^{-1}d^{-1}cd^2cd^{-1}cb^{-1}c^{-1}=1,\\
& & abc^{-1}b^{-1}c^{-1}aba^{-2}b=1,\\
& & d^{-2}cb^{-1}a^{-2}ba^{-2}d^{-2}c^{-1}=1,\\
& & c^{-1}aba^{-2}d^{-1}cd^{2}a^{3}bc^{-2}d=1,\\
& & b^{-1}a^{-1}b^{-1}d^{-1}cd^2aba^{-2}d^{-3}cb^{-1}=1\rangle.
\end{eqnarray*}
Let $Y=\hyp^3/\Gamma_1$.
There is a surjective group homomorphism $\Gamma_1\rar \GL_2(\F_{37})$, given by
\begin{eqnarray*}
a\mapsto \begin{pmatrix}11&1\\28&26 \end{pmatrix},\;\;\;
b\mapsto \begin{pmatrix}5&7\\30&32 \end{pmatrix},\;\;\;
c\mapsto \begin{pmatrix}8&20\\35&29 \end{pmatrix},\;\;\;
d\mapsto \begin{pmatrix}20&6\\26&17 \end{pmatrix},
\end{eqnarray*}
which gives rise to a $G$-covering $X\rar Y$
of arithmetic hyperbolic $3$-manifolds. A direct computation, as described in
Section \ref{sec:homcomp}, yields that
\begin{align*}
H_1(X/U_1,\Z) & \cong \Z^{14}\times C_2^2\times C_4^{12}\times C_8\times C_3\times
C_9^2\times C_{11} \times C_{1151}\times C_{11317},\\
H_1(X/U_2,\Z) & \cong H_1(X/U_1,\Z)\times C_{37},
\end{align*}
where $C_n$ denotes the cyclic group of order $n$. Here, we have
$\vol(X/U_1)=\vol(X/U_2) \approx 1106.067$.

\begin{figure}[tbh]
\centering
\includegraphics*[height=5cm,keepaspectratio=true]{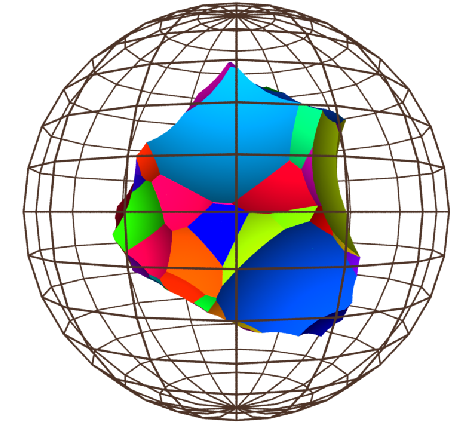}
\caption{A fundamental domain of the Kleinian group~$\Gamma_1$}\label{pic37}
\end{figure}

It follows from Corollary \ref{cor:impliesconj}, that the multiplicity
of $I=\rho(\chi,\triv)$ (see Section \ref{sec:regconst} for the definition)
in the $G$-module $H_1(X,\Q)$ is odd.
In this particular case, we computed directly that this multiplicity is 3.
\end{example}

\begin{example}
The following example shows that the $p$-torsion subgroups of two isospectral
manifolds may have the same order, but be non-isomorphic.

Let $G=\GL_2(\F_{19})$. Let $F$ be the ATR number field generated by a root~$t$ of
$x^4-x^3-3x-1$, and let $B$ be the Kleinian quaternion algebra
$$
B=\langle 1,i,j,k \mid i^2=j^2=k^2=ijk=-1\rangle_F.
$$
Let $\ideal$ be the principal ideal of $\Z_F$ generated by
$-t^3+t^2+2t+2$, of norm~$44$, and let $\Gamma_2=\Gamma(\order_0(\ideal))$,
as in Section \ref{sec:Y}. As in the previous example, $\Gamma_2$
is well-defined up to conjugacy in $\PSL_2(\C)$. 
A fundamental domain for~$\Gamma_2$ is displayed in Figure~\ref{pic19}.
The hyperbolic manifold
$Y=\hyp^3/\Gamma_2$
has a $G$-covering $X\rar Y$ (which we do not give explicitly this time),
such that
$$
H_1(X/U_1,\Z)[19^\infty]\cong C_{19} \times C_{19},\;\;\;H_1(X/U_2,\Z)[19^\infty]\cong C_{19^2}.
$$
Here, $\vol(X/U_1)=\vol(X/U_2)\approx 985.386$.

\begin{figure}[tbh]
\centering
\includegraphics*[height=5cm,keepaspectratio=true]{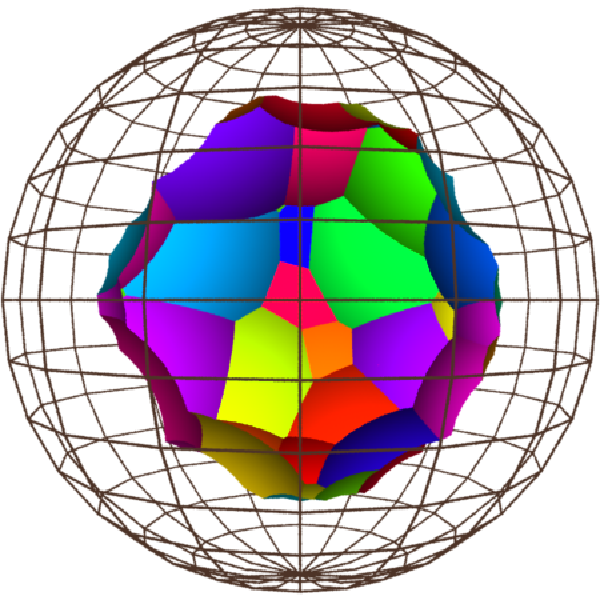}
\caption{A fundamental domain of the Kleinian group~$\Gamma_2$}\label{pic19}
\end{figure}

Corollary \ref{cor:impliesconj} implies that the multiplicity of $I$
in the $G$-module $H_1(X,\Q)$ is even, and in fact, in this example it is 0.
Moreover, unlike in the previous example, the non-isomorphic
torsion cannot be detected by the regulator constant of the $G$-module
$H_1(X,\Z)$.
\end{example}

\begin{example}
In the final example, we exhibit isospectral manifolds whose size
of torsion homology differs by a square. Such an example has interesting
representation theoretic features, as we will explain below.

Let $G=\GL_2(\F_{5})$, and let $\Theta=U_1-U_2$ be as in Example \ref{ex:p-Gassman}.
Let $F$ be the ATR number field generated by a root~$t$ of
$x^4-x^3-2x-1$, and let $B$ be the Kleinian quaternion algebra
$$
B=\langle 1,i,j,k \mid i^2=j^2=k^2=ijk=-1\rangle_F.
$$
\begin{figure}[tbh]
\centering
\includegraphics*[height=5cm,keepaspectratio=true]{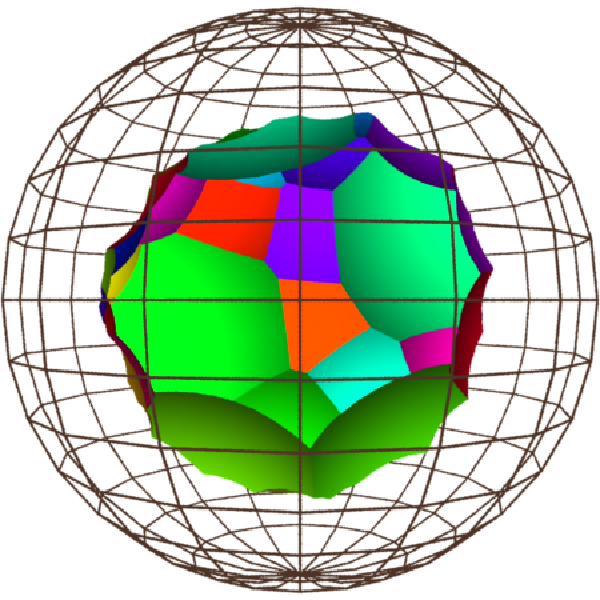}
\caption{A fundamental domain of the Kleinian group~$\Gamma_3$}\label{pic5}
\end{figure}

Let $\ideal$ be the principal ideal of $\Z_F$ generated by
$-t^2+2$, of norm~$23$, and let~$\Gamma_3=\Gamma(\order_0(\ideal))$,
as in Section \ref{sec:Y}. As in the previous examples, $\Gamma_3$
is well-defined up to conjugacy in $\PSL_2(\C)$.
A fundamental domain for~$\Gamma_3$ is displayed in Figure~\ref{pic5}.
The hyperbolic manifold
$Y=\hyp^3/\Gamma_3$
has a $G$-covering~$X\rar Y$ such that
$$
H_1(X/U_1,\Z)[5^\infty]\cong C_{5}^3,\;\;\;H_1(X/U_2,\Z)[5^\infty]\cong C_{5}.
$$
Here, $\vol(X/U_1)=\vol(X/U_2)\approx 223.790$.

Since the quotient $\#H_1(X/U_1,\Z)_{\tors}/\#H_2(X/U_1,\Z)_{\tors}$ is a square,
Corollary \ref{cor:modsquares} together with equation (\ref{eq:CheegerMuller})
imply that the regulator constants
$\cC_{\Theta}(H_2(X,\Q))\equiv\cC_{\Theta}(H_1(X,\Q))$ are trivial modulo
squares, which we confirm by computing that the multiplicity of~$I$
in~$H_1(X,\Q)$ is~$2$.
In particular, the non-triviality of $\#H_1(X/\Theta)_{\tors}$ cannot be detected
in the rational homology of $X$. However, by Corollary
\ref{cor:regcomparison}, the non-triviality of $\#H_1(X/\Theta,\Z)_{\tors}$
is detected in the $G$-module structure of the integral homology via the equality
$\ord_5(\cC_{\Theta}(H_2(X,\Z)))=2$.
\end{example}

We can obtain similar examples of interplay between the homological torsion and
representation theory for~$p=2$ using Proposition~\ref{prop:2regcst}.


\begin{thebibliography}{10}
%
\bibitem{Bartel}
A. Bartel, On Brauer--Kuroda type relations of $S$--class numbers in
dihedral extensions, J. Reine Angew. Math. \textbf{668} (2012), 211--244.


\bibitem{equivsurg}
A. Bartel and A. Page, Group representations in the homology of $3$-manifolds,
arXiv:1605.04866 [math.GT] (2016).

\bibitem{Benson}
D. J. Benson, Representations and Cohomology, Vol. I, Cambridge Studies
in Advanced Mathematics \textbf{30}, Cambridge University Press (1995).

\bibitem{BVtorsion}
N. Bergeron and A. Venkatesh, The asymptotic growth of torsion homology for
arithmetic groups, J. Inst. Math. Jussieu no. 2 (2013), 391--447.

\bibitem{BSVtorsion}
N. Bergeron, M. H. \c Seng\" un and A. Venkatesh, Torsion homology growth and
cycle complexity of arithmetic manifolds, Duke Math. J. \textbf{165} no. 9
(2016), 1629--1693.

\bibitem{Boltje}
R. Boltje, Class group relations from Burnside ring idempotents,
J. Number Theory \textbf{66} (1997), 291--305.


\bibitem{magma}
W. Bosma, J. Cannon and C. Playoust. The Magma algebra system. I. The user
language. J. Symbolic Comput., \textbf{24}  (1997) (3-4):235--265.

\bibitem{brockdunfield}
J. F. Brock and N. M. Dunfield, Injectivity radii of hyperbolic integer homology
$3$-spheres, Geom. Topol. \textbf{19} no. 1 (2015), 497--523.

\bibitem{Brown}
  K. S. Brown, Cohomology of Groups, GMT \textbf{87}, Springer (1982).

\bibitem{torsionJL}
F. Calegari and A. Venkatesh, A torsion Jacquet--Langlands correspondence,
arXiv:1212.3847v1 [math.NT] (2012).

\bibitem{Cheeger}
J. Cheeger, Analytic torsion and the heat equation, Ann. of Math. \textbf{109}
(1979), 259--322.

\bibitem{CR}
C. W. Curtis and I. Reiner, Methods of Representation Theory, with Applications to
Finite Groups and Orders, Vol. 2, John Wiley and Sons (1987).

\bibitem{tamroot}
T. Dokchitser and V. Dokchitser, Regulator constants and the parity
conjecture, Invent. Math. \textbf{178} no. 1 (2009), 23--71.

\bibitem{gordon}
C. Gordon, Survey of isospectral manifolds. Handbook of differential geometry,
Vol. I (2000), 747--778.

\bibitem{GordonSchueth}
C. Gordon, P. Perry and D. Schueth,
Isospectral and isoscattering manifolds: a survey of techniques and examples,
Geometry, spectral theory, groups, and dynamics,
Contemp. Math. \textbf{387}, Amer. Math. Soc. (2005), 157--179.

\bibitem{Hatcher}
A. Hatcher, Algebraic Topology, Cambridge University Press (2002).

\bibitem{Holt}
D. F. Holt, B. Eick, and E. A. O'Brien,
Handbook of Computational Group Theory,
Discrete Mathematics and its Applications (Boca Raton).
Chapman \& Hall (2005).

\bibitem{Ikeda}
A. Ikeda,
On the spectrum of a Riemannian manifold of positive constant curvature,
Osaka J. Math. \textbf{17} no. 1 (1980), 75--93.

\bibitem{Kac}
M. Kac, Can one hear the shape of a drum? Amer. Math. Monthly \textbf{73} (1966), 1--23.

\bibitem{lmfdb}
The LMFDB Collaboration, The L-functions and Modular Forms Database (2015),
\url{http://www.lmfdb.org}.

\bibitem{MaclachlanReid}
C. Maclachlan and A. W. Reid, The Arithmetic of Hyperbolic 3-Manifolds,
Graduate Texts in Mathematics \textbf{219}, Springer (2003).

\bibitem{Mueller1}
W. M\"uller, Analytic torsion and R-torsion of Riemannian manifolds, Adv. Math.
\textbf{28} (1978), 233--305.

\bibitem{Mueller2}
W. M\"uller, Analytic torsion and R-torsion for unimodular representations,
J. Amer. Math. Soc. \textbf{6} (1993), 721--753.

\bibitem{klngps}
A. Page, Computing arithmetic Kleinian groups,
Math. Comp. \textbf{84} no.~295 (2015)  2361--2390.

\bibitem{paridb}
The PARI~Group, PARI/GP version {\tt 2.7.5}, Bordeaux (2015),
\url{http://pari.math.u-bordeaux.fr/}.

\bibitem{Per1}
G. Perelman, The entropy formula for the Ricci flow and its geometric applications,
arXiv:math.DG/0211159 (2002).

\bibitem{Per2}
G. Perelman, Ricci flow with surgery on three-manifolds. arXiv:math.DG/0303109 (2003).

\bibitem{Per3}
G. Perelman,
Finite extinction time for the solutions to the Ricci flow on certain three-manifolds,
arXiv:math.DG/0307245 (2003).

\bibitem{PShapiro}
I. Piatetski-Shapiro, Complex Representations of $\GL(2,K)$ for Finite
Fields $K$, Contemporary Mathematics \textbf{16}, AMS (1983).

\bibitem{RaySinger}
D. B. Ray and I. M. Singer, $R$-torsion and the Laplacian on Riemannian manifolds,
Advances in Math. \textbf{7} (1971), 145--210.

\bibitem{maxorders}
I. Reiner, Maximal Orders, Academic Press (1975).

\bibitem{Laplace}
  S. Rosenberg, The Laplacian on a Riemannian Manifold, LMSST \textbf{31}, Cambridge
University Press (1997).

\bibitem{Shafarevich}
I. R. Shafarevich, Basic Algebraic Geometry, Springer (1974).

\bibitem{deSmit1}
B. de Smit, Generating arithmetically equivalent number fields with elliptic
curves in: J. P. Buhler (Ed.), Algorithmic Number Theory,
Lecture Notes in Computer Science 1423, Springer (1998) 392--399.

\bibitem{Sunada}
T. Sunada, Riemannian coverings and isospectral manifolds,
Annals Math. \textbf{121} no.~1 (1985), 169--186.

\bibitem{ClNoFormula}
J. Tate, Les conjectures de Stark sur les fonctions $L$ d'Artin en $s=0$,
Progress in Mathematics \textbf{47}, Birkhäuser (1984).

\bibitem{vigneras}
M.-F. Vign\'eras, Vari\'et\'es riemanniennes isospectrales et non isométriques,
Ann. of Math. (2) \textbf{112} no.1 (1980), 21--32.

\bibitem{quatalg}
J. Voight, Identifying the matrix ring: algorithms for quaternion algebras and
quadratic forms. Quadratic and higher degree forms, Dev. Math.
\textbf{31}, Springer (2013), 255--298.

\end{thebibliography}
\end{document}